\documentclass[10pt]{amsproc}
\usepackage{amssymb}
\usepackage{graphicx,color}
\usepackage{amsmath}
\usepackage{amscd} 

\theoremstyle{plain}
 \newtheorem{thm}{Theorem}[section]
 \newtheorem{prop}{Proposition}[section]
 \newtheorem{lem}{Lemma}[section]
 
 \newtheorem{claim}{Claim}[section]
\theoremstyle{definition}
 
 \newtheorem{dfn}{Definition}[section]
\theoremstyle{remark}
 \newtheorem{rem}{Remark}[section] 
 \numberwithin{equation}{section}

\renewcommand{\leq}{\leqslant}
\renewcommand{\geq}{\geqslant}
\renewcommand{\setminus}{\smallsetminus}
\title[Sums of two homogeneous Cantor sets]
{Sums of two homogeneous Cantor sets}

\author[Y.\ Takahashi]{YUKI TAKAHASHI}

\address{Department of Mathematics, Bar-Ilan University, Ramat-Gan, 5290002, Israel}

\email{takahashi@math.biu.ac.il}

\thanks{Y.\ T. \ was supported in part by NSF grant DMS-1301515 (PI: A.\ Gorodetski) and by the Israel Science Foundation grant 396/15 (PI: B.\ Solomyak).}

\date{today}

\begin{document}

\vspace{18mm}
\setcounter{page}{1}
\thispagestyle{empty}

\begin{abstract}
We show that for any two homogeneous Cantor sets with sum of Hausdorff dimensions that exceeds 1, one can create an interval in the sumset by applying arbitrary small perturbations (without leaving the class of homogeneous Cantor sets). 
In our setting the perturbations have more freedom than in the setting of the Palis' conjecture, so 
our result can be viewed as an affirmative answer to a weaker form of the Palis' conjecture. 
We also consider self-similar sets with overlaps on the real line (not necessarily homogeneous), and show that 
one can create an interval by applying arbitrary small perturbations, if the uniform self-similar measure has $L^2$-density.  
\end{abstract}
\maketitle

\section{Introduction and main results}
\subsection{Introduction}
Arithmetic sums and differences of two dynamically defined Cantor sets have been considered in many papers and in many different settings 
(e.g., \cite{Anisca}, \cite{Astels}, \cite{DG1}, \cite{DG}, \cite{Dekking}, \cite{Eroglu}, \cite{GN}, \cite{Hall}, \cite{Honary}, \cite{Mendes}, \cite{Mora}, \cite{Moreira00}, \cite{Moreira}, \cite{Moreira2}, 
\cite{Moreira3}, \cite{Nazarov}, \cite{Solomyak}, \cite{Solomyak1997}). 
It arises naturally in dynamical systems (e.g., \cite{PalisTakens}), in number theory (e.g., \cite{Astels}, \cite{Hall}), and in spectral theory (e.g., \cite{DG1}, \cite{DG}). 
It also has natural connection to the study of intersections of Cantor sets 
(e.g., \cite{Hunt}, \cite{KP}, \cite{Kraft}, \cite{Kraft3}, \cite{Moreira0}, \cite{Newhouse}). 

Motivated by questions in smooth dynamics, Palis conjectured that for generic pairs of dynamically defined 
Cantor sets either their arithmetic sum has zero Lebesgue measure or else it contains an interval (see, for example, \cite{PalisTakens}). 
For nonlinear case this was proven in \cite{Moreira3}. The problem is still open for affine Cantor sets. 
Even for the case of middle-$\alpha$ Cantor sets this question has not yet completely settled. 
Let us denote the middle-$\alpha$ Cantor set whose convex hull is $[0, 1]$ by $C_a$, where $a = \frac{1}{2}( 1 - \alpha )$. 
It is known that if the sum of the Hausdorff dimensions of two Cantor sets is smaller than $1$ then the arithmetic sum is a Cantor set 
of zero Lebesgue measure (see Proposition 1 in section 4 from \cite{PalisTakens}). Therefore, if 
\begin{equation*}
\frac{\log2}{ \log \left(1/a \right) } + \frac{\log2}{ \log \left( 1/ b \right) } < 1, 
\end{equation*}
then $C_a + C_b$ is a Cantor set. On the other hand, if 
\begin{equation*}
\frac{a}{1 - 2a} \cdot \frac{b}{1 - 2b} \geq 1,
\end{equation*}
then by Newhouse's Gap Lemma (see section 4.2 in \cite{PalisTakens}), the set $C_a + C_b$ is an interval. 
This still leaves a ``mysterious region" $\mathcal{R}$ where the morphology of the sumset is unclear. See Figure \ref{mysterious}. 
In \cite{Solomyak1997}, Solomyak showed that for a.e. $(a, b) \in \mathcal{R}$ the set $C_a + C_b$ has positive Lebesgue measure, 
and in \cite{Pourbarat2} Pourbarat constructed a nonempty open set contained in $\mathcal{R}$ such that if $(a, b)$ belongs to this 
open set then $C_a + C_b$ has (persistently) nonempty interior. 
It is still an open question whether the sum contains an interval for a.e. $(a, b) \in \mathcal{R}$.

In this paper we show that for any two homogeneous Cantor sets  
one can create an interval in the sumset by applying arbitrary small perturbations, if 
the sum of their Hausdorff dimensions is greater than $1$. 
In our setting the perturbations have more freedom than in the setting of the Palis' conjecture, so 
our result can be viewed as an affirmative answer to a weaker form of the Palis' conjecture.  
We rely heavily on the techniques invented by Moreira and Yoccoz in \cite{Moreira3}. 
Their proof is very technical and involved but it can be simplified significantly in the context of this paper, 
while it still contains many of the key ideas in \cite{Moreira3}. 
One big difference is that we do not require the Scale Recurrence Lemma (see section 3 in \cite{Moreira3}), which is a deep result about 
the relative sizes of nonlinear Cantor sets under renormalization operations. 
\begin{centering}
\begin{figure}[t]
\includegraphics[scale=1.00]{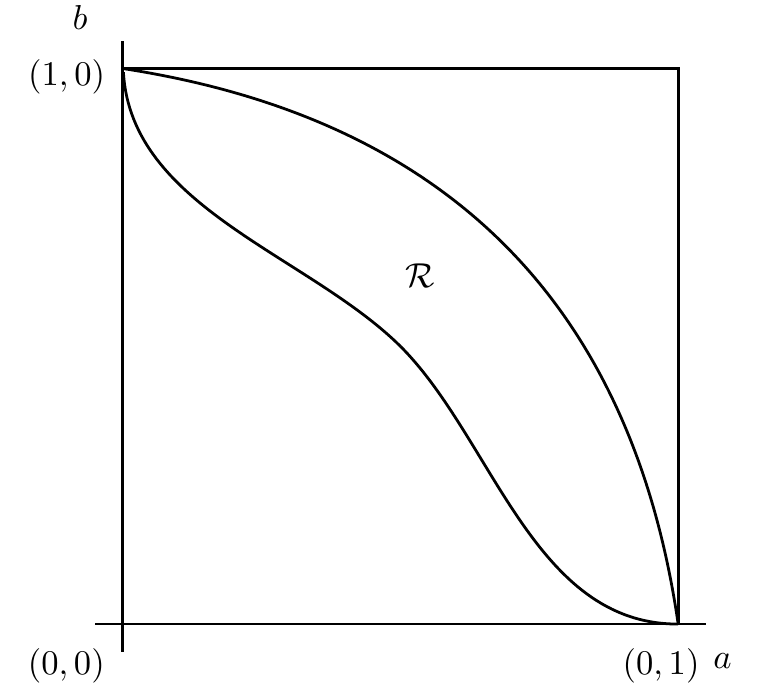}
\caption{``mysterious region" $\mathcal{R}$}
\label{mysterious}
\end{figure}
\end{centering}

The basic idea of the proof is to create a recurrent set. A set $\mathcal{L} \subset \mathbb{R}$ 
is a recurrent set if any point in $\mathcal{L}$ 
can go back to ``well inside $\mathcal{L}$" by the action of a renormalization operator. 
In our settings, recurrent set is 
a set of relative positions of the given two Cantor sets, and  
can be viewed as a compact subset of $\mathbb{R}$ 
(in \cite{Moreira3} recurrent set is more complicated due to the nonlinearity of Cantor sets). 
To create a recurrent set, we first construct a set which is ``very likely" to be a recurrent set and make it really a recurrent set by applying a small perturbation. 
Let us remark that in some papers 
recurrent sets are explicitly constructed for some affine Cantor sets \cite{Honary}, \cite{Pourbarat}, \cite{Pourbarat2}. 
We will also show that analogous result holds for sums of homogeneous Cantor sets with itself.   
Note that our results do not follow from \cite{Moreira3}. This is because (at least) one Cantor set has to be nonlinear in order to 
apply the Scale Recurrence Lemma. 
Let us also remark that in our setting considering sums of two homogeneous Cantor sets is essentially equivalent to considering 
self-similar sets with overlaps on the real line. 
Questions on self-similar sets and measures with overlaps have been considered in many papers 
(e.g. \cite{Hochman}, \cite{PS}, \cite{Peres}, \cite{Shmerkin0}, \cite{Shmerkin}, \cite{SS}, \cite{Varju}) and are known to be extremely difficult. 
We show that for any self-similar sets with overlaps   
one can create an interval by applying arbitrary small perturbations, if the uniform self-similar measure has $L^2$-density.

\subsection{Homogeneous Cantor sets and main results}\label{section2}

For any compact set $A \subset \mathbb{R}$, we denote the convex hull of 
$A$ by $\mathrm{con}(A)$. 

\begin{dfn}\label{selfsimilar}
A set $K \subset \mathbb{R}$ is \emph{self-similar} if the following holds: 
there exists a finite alphabet $\mathcal{A}$ and a set of linear contractions $\mathcal{F} = \{ f_a \}_{a \in \mathcal{A}}$ 
on $\mathbb{R}$ 
such that 
\begin{equation*}
K = \displaystyle{ \bigcup_{a \in \mathcal{A}} f_a(K) }.
\end{equation*} 
If, in addition, each $f_a \ (a \in \mathcal{A})$ has the same contracting ratio $0 < \rho < 1$ and 
$f_a( \mathrm{con}(K) ) \ (a \in \mathcal{A})$ are pairwise disjoint, we call $K$ a \emph{homogeneous Cantor set}. 

We denote $f_a( \mathrm{con}(K) )$ by $I(a)$. 
Similarly, for $a_1, a_2 \in \mathcal{A}$ we denote $(f_{a_1} \circ f_{a_2}) ( \mathrm{con}(K) )$ by $I(a_1 a_2)$. 
We call $\mathcal{F}$ a \emph{set of $\rho$-contractions} associated to $K$. 
\end{dfn}

We denote the set of homogeneous Cantor sets and the set of self-similar sets by $\mathcal{H}$, $\mathcal{H}_0$, respectively. 

\begin{rem}
Definition \ref{selfsimilar} is not the most standard. 
Self-similar set is normally defined as a set $K$ together with a set of contracting maps which generates $K$. 
\end{rem}

It is easy to see that for any $n \in \mathbb{N}$, 
homogeneous Cantor set $K$ can also be generated by the set of contractions 
\begin{equation*}
\left\{  f_{a_1} \circ \cdots \circ f_{a_n} : a_1, \cdots, a_n \in \mathcal{A}  \right\}. 
\end{equation*}
Therefore, we can assume that $\rho > 0$ is arbitrary small. 

Let $K$, $K'$ be homogeneous Cantor sets. We denote by 
$d$ (resp. $d'$) the Hausdorff dimension of $K$ (resp. $K'$). 
Let $\mu$ (resp. $\mu'$) be the uniform self-similar probability measure associated to $K$ (resp. $K'$).

\begin{dfn}
Let $\mathcal{Y}$ be the set of pairs of homogeneous Cantor sets $(K, K')$ such that 
\begin{itemize}
\item[(i)] $\mu \ast \mu'$ has $L^2$-density;
\item[(ii)] $d + d' > 1$; 
\item[(iii)] there exist $\rho$-contractions $\mathcal{F}$ (resp. $\mathcal{F}'$) associated to $K$ (resp. $K'$) for some $0 < \rho < 1$  
(both $\mathcal{F}$ and $\mathcal{F}'$ are sets of contractions whose contracting ratio is $\rho$).  
\end{itemize} 
\end{dfn}

\begin{rem}
The first condition in the above definition is quite mild.  
Kaufman's proof of the Marstrand' theorem tells us that,  
under the condition of $d + d' > 1$, the measure $\mu \ast \lambda \mu'$ has $L^2$-density for Lebesgue almost all $\lambda$ 
(see, for example, section 4.2 in \cite{PalisTakens}).  
\end{rem}

\begin{rem}
The third condition is also not too restrictive. 
Assume that $\mathcal{F}$ (resp. $\mathcal{F}'$) is a set of $\lambda$- (resp. $\lambda'$-) contractions associated to $K$ (resp. $K'$). 
Then, it is easy to see that the third condition is equivalent to 
\begin{equation*}
\frac{ \log \lambda }{ \log \lambda' } \in \mathbb{Q}, 
\end{equation*}
and this assumption can be satisfied by applying arbitrary small perturbations. 
\end{rem}

For any two intervals $J, \widetilde{J} \subset \mathbb{R}$ with $| J | = | \widetilde{J} |$, 
we denote $|t| / |J|$ by $\Delta( J, \widetilde{J} )$, where $t$ is the real number 
which satisfies $J + t = \widetilde{J}$. We put a topology on $\mathcal{H}$ in the following way:

Let $K \in \mathcal{H}$, and let $\mathcal{U}_{\epsilon, K}$ be the set of all 
$\widetilde{K} \in \mathcal{H}$ such that the following holds: there exist sets of contractions $\mathcal{F} = \{f_a\}_{a \in \mathcal{A}}$
 (resp. $\widetilde{\mathcal{F}} = \{\tilde{f}_{\tilde{a}}\}_{\tilde{a} \in \widetilde{ \mathcal{A} } }$) 
associated to $K$ (resp. $\widetilde{K}$) such that 
\begin{itemize}
\item[(i)] $|\mathrm{con}(K)| = | \mathrm{con}( \widetilde{K} ) |$; 
\item[(ii)] $\mathcal{A} = \widetilde{\mathcal{A}}$; 
\item[(iii)] $| I(a) | = | \widetilde{I}(a) |$ for all $a \in \mathcal{A}$; 
\item[(iv)] $\Delta( I(a), \widetilde{I}(a) ) \leq \epsilon$ for all $a \in \mathcal{A}$. 
\end{itemize}
We consider the topology on $\mathcal{H}$ generated by 
$\left\{ \mathcal{U}_{\epsilon, K} \mid \epsilon > 0, K \in \mathcal{H} \right\}$. 
By verbatim repetition, we define a topology on $\mathcal{H}_0$.  

\vspace{1mm}

Our main results are the following: 

\begin{thm}\label{thm1}
Let $(K, K') \in \mathcal{Y}$. 
Then, for every $\epsilon > 0$, there exists $\widetilde{K} \in \mathcal{H}$ such that 
\begin{itemize}
\item[(i)] $\widetilde{K}$ is $\epsilon$-close to $K$; 
\item[(ii)] $\widetilde{K} + K'$ contains an interval. 
\end{itemize}
\end{thm}

Analogous result holds for sums of Cantor sets with itself.

\begin{thm}\label{thm2}
Let $(K, K) \in \mathcal{Y}$. 
Then, for every $\epsilon > 0$, 
there exists $\widetilde{K} \in \mathcal{H}$ such that 
\begin{itemize} 
\item[(i)] $\widetilde{K}$ is $\epsilon$-close to $K$;
\item[(ii)] $\widetilde{K} + \widetilde{K}$ contains an interval. 
\end{itemize}
\end{thm}

We also obtain a similar result for self-similar sets with overlaps.  
The proof is essentially the same as Theorem \ref{thm1}. 

\begin{dfn}
Let $K \in \mathcal{H}_0$, and let $\{f_a\}_{a \in \mathcal{A}}$ be a set of contractions which generates $K$. 
Let $0 < r_a < 1$ be the contracting ratio of $f_a$, and let $s$ be the unique solution of $\sum_{a \in \mathcal{A}} r^s_a = 1$. 
We call a probability measure $\mu$ the \emph{uniform self-similar measure} if $\mu$ satisfies 
\begin{equation*}
\mu = \sum_{a \in \mathcal{A}} r^s_a f_a \mu. 
\end{equation*}
\end{dfn}

\begin{rem}
As with Definition \ref{selfsimilar}, 
The above definition is not the most standard. 
Uniform self-similar measure is normally defined for a set $K$ together with a set of contracting maps which generates $K$. 
\end{rem}

Then we have the following: 

\begin{thm}\label{thm3}
Let $K \in \mathcal{H}_0$, and let us assume that the uniform self-similar measure of $K$ has $L^2$-density. 
Then, for every $\epsilon > 0$, there exists $\widetilde{K} \in \mathcal{H}_0$ such that 
\begin{itemize} 
\item[(i)] $\widetilde{K}$ is $\epsilon$-close to $K$;
\item[(ii)] $\widetilde{K}$ contains an interval. 
\end{itemize}
\end{thm}

\subsection{Structure of the paper}
We rely heavily on the techniques invented in \cite{Moreira3}. 
Sections 2, 3, 4 and 5 in this paper correspond to the sections 2, 3, 4 and 5 in \cite{Moreira3}. 
We describe the basic ideas of the proof in section \ref{recurrence}. 
The outline of the proof of Theorem \ref{thm1} is given in section \ref{outline}. 
In section \ref{construct} we will construct the sets $\mathcal{L}^0(\rho)$ for any $\rho > 0$ which are candidates of a recurrent set, and show that 
the Lebesgue measure of $\mathcal{L}^0(\rho)$ are bounded away from zero uniformly in $\rho > 0$.  
In section \ref{key_prop} we will prove the key proposition, which claims that with ``very high probability" any point in the set $\mathcal{L}^0(\rho)$  
can return to itself by an action of a renormalization operator. 
Theorem \ref{thm2} and \ref{thm3} will be proven in section \ref{proof2} (the proof is very similar to Theorem \ref{thm1}).

\section{Intersections of homogeneous Cantor sets and recurrent sets}\label{recurrence}

\subsection{Differences of Cantor sets}
Let $K$, $K'$ be homogeneous Cantor sets. Since $K + K' = K - (-K')$ and $-K'$ is again a homogeneous Cantor set, 
from below we consider only differences of homogeneous Cantor sets, instead of sums. 

Notice that, since 
\begin{equation*}
t \in K - K' \iff K \cap ( K' + t ) \neq \phi, 
\end{equation*}
$K - K'$ contains an interval if and only if there exists an interval $J \subset \mathbb{R}$ such that 
\begin{equation*}
K \cap ( K' + t ) \neq \phi \, \text{ for all } t \in J. 
\end{equation*}

\subsection{Recurrent sets}
Throughout this section, we fix $(K, K') \in \mathcal{Y}$  
and associated $\rho$-contractions $\mathcal{F} = \{ f_a \}_{a \in \mathcal{A}}$ and 
$\mathcal{F}' = \{ f'_{a'} \}_{a' \in \mathcal{A}' }$. 
Without loss of generality, we can assume that $\mathrm{con}(K) = [0, 1]$ and $\mathrm{con}(K') = [0, s_0]$. 

Let $\mathcal{P}$ (resp. $\mathcal{P}'$) be the set of all affine maps with positive linear coefficient  
defined on $\mathrm{con}(K)$ (resp. $\mathrm{con}(K')$). 
We call a pair $(h, h') \in \mathcal{P} \times \mathcal{P}'$ a \emph{configuration} of $K$, $K'$ 
if $h$ and $h'$ have the same linear coefficient. 
We define a equivalence relation on the set of configurations in the following way: 
\begin{equation*}
\begin{aligned}
&(h_1, h_1') \sim  (h_2, h'_2) \\
\iff &\text{there exists an affine map $g$ such that } g \circ h_1 = h_2 \text{ and } g \circ h'_1 = h'_2. 
\end{aligned}
\end{equation*}
Let $Q$ be the quotient of the configurations by this equivalence relation. 
We call an element of $Q$ a \emph{relative configuration} of $K$, $K'$.

Let $u \in Q$, and let 
$(h, h') \in \mathcal{P} \times \mathcal{P}'$ be the configuration such that $u = [(h, h')]$ and 
$h(\mathrm{con}(K)) = [0, 1]$. Consider the map  
\begin{equation}\label{identification0}
\begin{aligned}
Q &\to \mathbb{R} \\
u &\mapsto h'(K'^L)
\end{aligned}
\end{equation}
where $K'^L$ is the left endpoint of $K'$. It is easy to see that this map is a bijection. 
By abuse of notation, from below we simply denote $h'(K'^L)$ by $u$. 
For $a \in \mathcal{A}$ and $a' \in \mathcal{A}'$, we define a map $T_{a} T'_{a'}(\cdot): Q \to Q$ by 
\begin{equation}\label{renormalization}
T_{a} T'_{a'} ( [(h, h')] ) = [ ( h \circ f_{a}, h' \circ f'_{a'} ) ], 
\end{equation}
and call this map $T_{a} T'_{a'}(\cdot)$ a \emph{renormalization operator}.  
Similarly, for $a_1, a_2 \in \mathcal{A}$ and $a'_1, a'_2 \in \mathcal{A}'$, we define a map $T_{a_1 a_2} T'_{a'_1 a'_2}(\cdot): Q \to Q$ by 
\begin{equation*}
T_{a_1 a_2} T'_{a'_1 a'_2} ( [(h, h')] ) = [ ( h \circ f_{a_1} \circ f_{a_2}, h' \circ f'_{a'_1} \circ f'_{a'_2} ) ], 
\end{equation*}
and call this also a renormalization operator. 
Note that we have $T_{a_1 a_2} T'_{a'_1 a'_2} = ( T_{a_2} T'_{a'_2} ) \circ ( T_{a_1} T'_{a'_1} )$. 

\begin{rem}
We can naturally define renormalization operator for any pair of words 
$\underline{a} = a_1 a_2 \cdots a_n$ and $\underline{a}' = a'_1 a'_2 \cdots a'_n$ with $n \geq 3$, 
but for our purpose the above definition suffices. 
\end{rem}

\begin{rem}
Let $u \in Q$, and let $(h, h')$ be a configuration such that $u = [(h, h')]$. 
Then $u$ represents the ``relative position" of two Cantor sets $h(K)$ and $h'(K')$. 
For any $a \in \mathcal{A}$ (resp. $a' \in \mathcal{A}'$), 
the Cantor set $h \circ f_a (K)$ (resp. $h' \circ f'_{a'}(K')$) is a subset of the Cantor set $h(K)$ (resp. $h'(K')$) 
which is associated to the word $a$ (resp. $a'$). The real number $T_a T'_{a'}(u)$ 
represents the ``relative position" of these two Cantor sets. 
\end{rem}


\begin{dfn}
Let $u \in Q$, and let $(h, h')$ be a configuration such that $u = [(h, h')]$. Then we say that $u$ is 
\begin{itemize}
\item[--] \emph{intersecting} if $h(K) \cap h'(K') \neq \phi$; 
\item[--] \emph{linked} if $h( \mathrm{con}(K) ) \cap h'( \mathrm{con}(K') ) \neq \phi$. 
\end{itemize} 
\end{dfn}
Recall that we identified $Q$ with $\mathbb{R}$ by (\ref{identification0}). 
Note that, if $u \in Q$ is linked, then $|u| \leq \max\{ 1, s_0 \} < 1 + s_0$.  
 
\begin{lem}\label{crucial}
Let $u \in Q$. Then $u$ is intersecting if and only if the following holds: 
there exist $M > 0$ and $a_i \in \mathcal{A}, \, a'_i \in \mathcal{A}' \ (i = 1, 2, \cdots)$ such that 
$| u_i  | < M \ (i = 0, 1, 2, \cdots)$, where $u_i  \ (i = 0, 1, \cdots)$ is a sequence defined recursively by 
\begin{equation}\label{u}
u_0 = u, \ u_{i} = T_{a_i} T'_{a'_i} u_{i-1}.
\end{equation} 
\end{lem}

\begin{proof}
Assume first that $u$ is intersecting. Let $(h, h')$ be a configuration such that $u = [(h, h')]$. 
Let $x, x' \in \mathbb{R}$ be such that $h(x) = h'(x')$, and let 
$a_i \in \mathcal{A}, \, a'_i \in \mathcal{A}' \ (i = 1, 2, \cdots)$ be the sequences such that 
\begin{equation*}
x = \bigcap_{i=1}^{\infty} f_{a_1} \circ \cdots \circ f_{a_i} (K) \text{ \ and \ } 
x' = \bigcap_{i=1}^{\infty} f'_{a'_1} \circ \cdots \circ f'_{a'_i} (K'). 
\end{equation*}
Define $\{ u_i \}$ by (\ref{u}). Since $u_i$'s are intersecting, they are linked. Therefore, we have 
$| u_i  | < 1 + s_0$. 

Assume next that $u$ is not intersecting. Let us take $a_i \in \mathcal{A}, \, a'_i \in \mathcal{A}' \ (i = 1, 2, \cdots)$, and let $\{ u_i \}$ be the sequence 
defined by (\ref{u}). Let $(h, h')$ be a configuration such that $u = [(h, h')]$. 
Note that $u_i = [ (h_i, h'_i) ]$, where $h_i = h \circ f_{a_1} \circ \cdots \circ f_{a_i}$ and 
$h'_i = h' \circ f'_{a'_1} \circ \cdots \circ f'_{a'_i}$. 
Write 
\begin{equation*}
x = \bigcap_{i=1}^{\infty} f_{a_1} \circ \cdots \circ f_{a_i} (K) \text{ \ and \ } 
x' = \bigcap_{i=1}^{\infty} f'_{a'_1} \circ \cdots \circ f'_{a'_i} (K'). 
\end{equation*}
Then, $h(x) \in h_i( \mathrm{con}(K) )$ and $h'(x') \in h'_i( \mathrm{con}(K') )$. Since $h(x) \neq h'(x')$ and 
 $\lim_{i \to \infty} | h_i( \mathrm{con}(K))| = 0,  \lim_{i \to \infty} | h'_i( \mathrm{con}(K') ) | = 0$, 
 we obtain $\lim_{i \to \infty} | u_i  | = \infty$. 
\end{proof}

The above lemma leads to the following definition. 
Moreira and Yoccoz used this idea in \cite{Moreira3} to consider stable intersections of dynamically defined (nonlinear) Cantor sets. 
See also \cite{Honary}, \cite{Pourbarat}, \cite{Pourbarat2}. 

\begin{dfn}
We call a compact set $\mathcal{L}$ in $Q$ a \emph{recurrent set} 
if for every $u \in \mathcal{L}$, 
there exist $a \in \mathcal{A}$ and $a' \in \mathcal{A}'$  
such that $T_{a} T'_{a'} u$ belongs to $\mathrm{int} \mathcal{L}$. 
\end{dfn}

Lemma \ref{crucial} implies the following: 

\begin{prop}
If a recurrent set contains an interval, then $K - K'$ contains an interval. 
\end{prop}



\subsection{Geometrical interpretation}
Let us denote the line $y = -x$ by $\ell$, and let $\pi$ be the orthogonal projection of $\mathbb{R}^2$ onto $\ell$. 
We parametrize $\ell$ by  
\begin{equation}\label{confusing}
\mathbb{R} \ni x \mapsto x 
\begin{pmatrix}
\frac{1}{2} \\
-\frac{1}{2}
\end{pmatrix}. 
\end{equation}
From below we always assume that $\ell$ has this parametrization. 
It is easy to see that 
$\pi ( K \times K' ) = K - K'$ under this identification.

Assume that we are given $(K, K') \in \mathcal{Y}$. 
Given $(x, x') \in K \times K'$, there is a unique relative configuration $u = [(h, h')]$ such that 
\begin{equation*}
h( x ) = h'( x' ). 
\end{equation*}
It is easy to see that the coordinate of this configuration $u$ is given by $x - x'$.

Let $a \in \mathcal{A}, a' \in \mathcal{A}'$, and consider $K \times K'$ in $\mathbb{R}^2$. 
Assume that $u \in Q$ is a relative configuration. 
Denote by $\ell_u$ the line which is perpendicular to $\ell$ and passes through $u \in \ell$ 
(we identified $u \in Q$ with a point in $\ell$ by (\ref{confusing})). 
Then,  
\begin{equation*}
\ell_u \cap ( I(a) \times I'(a') ) \neq \phi \iff u' \text{ is linked},
\end{equation*}
Where $u' = T_{a} T'_{a'}(u)$. See Figure \ref{geom00}. The coordinate of $u'$ can be seen graphically by considering the ``relative position between  
the rectangle $I(a) \times I'(a')$ and the line $\ell_u$". To be more precise, 
let us consider the affine transformation which sends the rectangle $I(a) \times I'(a')$ to the rectangle $\mathrm{con}(K) \times \mathrm{con}(K')$. 
This transformation sends the point $p$ in Figure \ref{geom00} to a point on the line $\ell$. This point agrees with $u'$ 
(under the identification (\ref{confusing})). 

Let us denote by $\mu_{K, K'}$ the push-forward of $\mu \times \mu'$ under the map $\pi$. 
Since we consider differences of Cantor sets, instead of $\mu \ast \mu'$ we need to assume that $\mu_{K, K'}$ has $L^2$-density. 
Notice that $\mu \times \mu'$ agrees with the self-similar probability measure associated to the set of contractions 
$\{f_a \times f'_{a'} \}_{a \in \mathcal{A}', a' \in \mathcal{A}'}$ on $\mathbb{R}^2$.

\begin{centering}
\begin{figure}[t]
\includegraphics[scale=1.00]{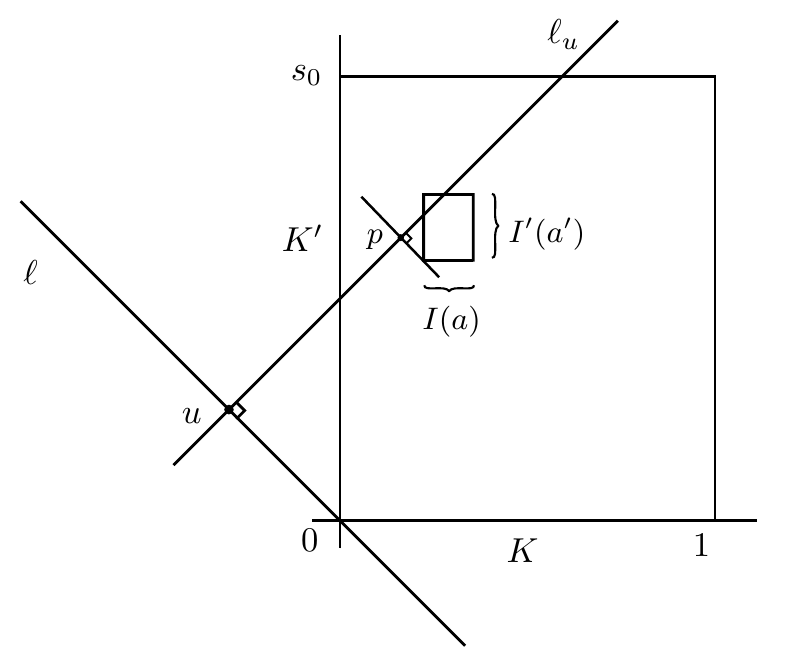}
\caption{}
\label{geom00}
\end{figure}
\end{centering}

\section{Outline of the proof of the main theorem}\label{outline}

\subsection{Perturbation}
In this section, we discuss outline of the proof of Theorem \ref{thm1}. 
Assume that we are given $(K, K') \in \mathcal{Y}$.  
Let $\mathcal{F} = \{ f_a \}_{a \in \mathcal{A}}$ (resp. $\mathcal{F}' = \{ f'_{a'} \}_{a' \in \mathcal{A}' } $) be a set of 
$\rho^{1/2}$-contractions associated to $K$ (resp. $K'$). 
We assume that $\rho > 0$ is sufficiently small, and 
$\mathrm{con}(K) = [0, 1]$ and $\mathrm{con}(K') = [0, s_0]$. 

In section \ref{construct} we will define disjoint subsets $\mathcal{A}_1$, $\mathcal{A}_2 \subset \mathcal{A}$. We only perturb the elements of the Markov partition 
associated to $\mathcal{A}_1$. Let $\Omega = [-1, 1]^{\mathcal{A}_1}$. 
For $\underline{\omega} \in \Omega$,  we define 
\begin{equation*}
\mathcal{F}^{\underline{\omega}} = \{ f^{\underline{\omega}}_a \}_{a \in \mathcal{A}}, 
\end{equation*}
a set of $\rho^{1/2}$-contractions, in the following way:  
\begin{equation*}
f^{\underline{\omega}}_a (x) = 
\begin{cases}
f_a(x) & \text{ if \ } a \in \mathcal{A} \setminus \mathcal{A}_1  \\
f_a(x) + c_0 \rho \, \underline{\omega}(a) & \text{ if \ } a \in \mathcal{A}_1, 
\end{cases}
\end{equation*}
where $c_0$ is a sufficiently large constant, to be chosen later. 
Let $K^{\underline{\omega}}$ 
be the homogeneous Cantor set associated to $\mathcal{F}^{\underline{\omega}}$. 
Recall that we defined the renormalization operators in (\ref{renormalization}). 
We define the renormalization operators 
$T^{\underline{\omega}}_{a} T'_{a'}(\cdot) \ (a \in \mathcal{A}, a' \in \mathcal{A}')$ 
in analogous way. 

\begin{rem}
In the proof we use constants $c_{k} \ (k = 0, 1, \cdots, 11)$. 
They may depend on each other but can be taken independently of $\rho > 0$. 
\end{rem}

\subsection{Outline of the proof}

In section \ref{construct}, we will construct a nonempty bounded set $\mathcal{L}^0(\rho) \subset \mathbb{R}$. 
Let $\mathcal{L}^1(\rho)$ (resp. $\mathcal{L}(\rho)$) be the $\rho$ (resp. $\rho/2$) neighborhood of $\mathcal{L}^0(\rho)$.  
We show that for some $\underline{\omega} \in \Omega$, $\mathcal{L}(\rho)$ is a recurrent set for $(K^{\underline{\omega}}, K')$. 

For $t \in \mathcal{L}^1(\rho)$, we define $\Omega^0(t) \subset \Omega$ to be the set of all 
$\underline{\omega} \in \Omega$ such that the following holds: 
there exist $a_1$, $a_2 \in \mathcal{A}$ and 
$a'_1$, $a'_2 \in \mathcal{A}'$, and the image 
\begin{equation*}
T^{\underline{\omega}}_{a_1 a_2} T'_{a'_1 a'_2}(t) 
= \hat{t}
\end{equation*}
satisfies $\hat{t} \in \mathcal{L}^0(\rho)$. 
The following crucial estimate will be proven in section \ref{key_prop}. 

\begin{prop}\label{key_lem}
There exists $c_1 > 0$ such that for any $t \in \mathcal{L}^1(\rho)$, 
\begin{equation*}
\mathbb{P} \left( \Omega \setminus \Omega^0(t) \right) \leq \exp \left( -c_1 \rho^{ -\frac{1}{2} (d + d' - 1) } \right). 
\end{equation*}
\end{prop}

We have not yet defined $\mathcal{A}_1, \mathcal{A}_2 \subset \mathcal{A}$ and $\mathcal{L}^{0}(\rho)$. 
They are constructed in such a way that Proposition \ref{key_lem} holds. Below we prove Theorem \ref{thm1} assuming 
Proposition \ref{key_lem}. 
In section \ref{construct1} we explain how to construct $\mathcal{L}^0(\rho)$ using $\mathcal{A}_1, \mathcal{A}_2$, and in section \ref{construct2} we construct 
the sets $\mathcal{A}_1, \mathcal{A}_2$ and show that the measure of the set $\mathcal{L}^{0}(\rho)$ is bounded away from 
zero uniformly in $\rho > 0$. Combining all these properties we prove Proposition \ref{key_lem} in section \ref{key_prop}. 

\begin{rem}
Proposition \ref{key_lem} says that, if $\rho > 0$ is sufficiently small then $\Omega^0(t)$ is ``almost the whole set $\Omega$". 
Therefore, if $t \in \mathcal{L}^{1}(\rho)$ then with ``very high probability'' 
$t$ can come back to ``well inside $\mathcal{L}^1(\rho)$" by the action of a renormalization operator. 
\end{rem}

\begin{rem}\label{daizi}
Here is a heuristic argument why Proposition \ref{key_lem} holds: 
the set $\mathcal{L}^{0}(\rho)$ is formed of the values of $t$ such that $|t| < 1 + s_0$ and such that 
$| T_{\underline{b}} T'_{\underline{b}'}(t) |$ is bounded by $1 + s_0$ for at least $c^2_{2} \rho^{-\frac{1}{2}(d + d'-1)}$ many pairs of 
$(\underline{b}, \underline{b}') \in \mathcal{A}^2 \times \mathcal{A}'^{2}$. 
Furthermore, $\mathcal{L}^0(\rho)$ has positive Lebesgue measure which is bounded away from zero uniformly in $\rho > 0$ 
(in fact $\mathcal{L}^{0}(\rho)$ is a union of intervals).  
Therefore, if $t \in \mathcal{L}^{0}(\rho)$ then it is reasonable to expect that 
for each such $( \underline{b}, \underline{b}') \in \mathcal{A}^2 \times \mathcal{A}'^{2}$, 
$t$ can return to $\mathcal{L}^{0}(\rho)$ by the action of the renormalization operator $T_{ \underline{b} } T'_{ \underline{b}'}$
with probability greater than $| \mathcal{L}^0( {\rho} ) | / 2 c_0$. Denote this value by $c'_1$. 
 Then,   
if $t \in \mathcal{L}^{0}(\rho)$ the probability that $t$ cannot return to $\mathcal{L}^{0}(\rho)$ by the action of renormalization operators is expected to be smaller than 
$(1 - c'_{1})^{c^2_{2} \rho^{-\frac{1}{2}(d + d'-1)}}$. 
Since the sizes of $\mathcal{L}^{0}(\rho)$ and $\mathcal{L}^{1}(\rho)$ are ``almost the same'', the same argument holds for $t \in \mathcal{L}^{1}(\rho)$. 
This explains Proposition \ref{key_lem}. The formal proof is given in section \ref{key_prop}. 
\end{rem}


The set $\mathcal{L}^1(\rho)$ satisfies  
\begin{equation*}
t \in \mathcal{L}^1(\rho) \implies |t| \leq 1 + s_0 + \rho. 
\end{equation*}
We choose a finite $\rho^{ 5/2 }$-dense subset $\Delta$ of $\mathcal{L}^1(\rho)$. 
Note that $\# \Delta \leq 2(1 + s_0) \rho^{ -5/2 }$. 
Now, if $\rho > 0$ is small enough, 
\begin{equation*}
2(1 + s_0) \, \rho^{-5/2} \exp \left( -c_1 \rho^{ -\frac{1}{2} ( d + d' - 1 ) } \right) < 1, 
\end{equation*}
and therefore we can find $\underline{\omega}_0 \in \Omega$ such that 
$\underline{\omega}_0 \in \Omega^0(t)$ for all $t \in \Delta$.

\begin{centering}
\begin{figure}[t]
\includegraphics[scale=0.88]{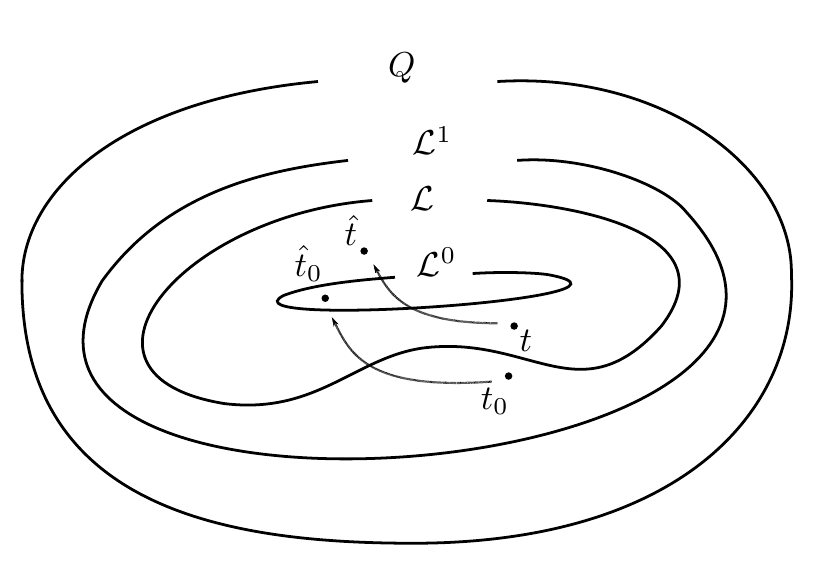}
\caption{Recurrent set}
\label{recurrence_figure}
\end{figure}
\end{centering}

\begin{rem}
This is saying that for any $t \in \Delta$, $t$ can return to $\mathcal{L}^0(\rho)$ by an 
action of the renormalization operator of the form 
$T^{\underline{\omega}_0}_{\underline{b}} T'_{\underline{b}'}$. 
\end{rem}

Theorem \ref{thm1} follows from the following claim: 

\begin{claim}
For $(K^{\underline{\omega}_0}, K')$, $\mathcal{L}(\rho)$ is a nonempty compact recurrent set. 
\end{claim}

\begin{proof}[proof of the claim]
Let $t \in \mathcal{L}(\rho)$. 
Let us take $t_0 \in \Delta$ such that $|t - t_0| < \rho^{ 5/2 }$. 
By the choice of $\underline{\omega}_0$, we have $\underline{\omega}_0 \in \Omega^{0}( t_0 )$.  
Therefore, there exist $\underline{b} \in \mathcal{A}^2$,  $\underline{b}' \in \mathcal{A}'^2$ such that, writing 
\begin{equation*}
T_{\underline{b}}^{\underline{\omega}_0} T'_{\underline{b}'}(t_0) = \hat{t}_0, 
\end{equation*}
we have $\hat{t}_0  \in \mathcal{L}^0(\rho)$. Let 
\begin{equation*}
T_{\underline{b}}^{\underline{\omega}_0} T'_{\underline{b}'} (t) = \hat{t}.
\end{equation*}
Since $| \hat{t} - \hat{t}_0 | = | t - t_0 | \cdot \rho^{-1} <  \rho^{5/2} \cdot \rho^{-1} = \rho^{ 3/2 } < \rho / 2$, 
we obtain $\hat{t} \in \mathcal{L}(\rho)$. 
\end{proof}

\section{Construction of the sets $\mathcal{A}_1$, $\mathcal{A}_2$ and $\mathcal{L}^0(\rho)$ }\label{construct}

\subsection{Construction of $\mathcal{L}^0(\rho)$}\label{construct1}
We first construct the set $\mathcal{L}^0(\rho)$. 
The construction involves $\mathcal{A}_1, \mathcal{A}_2 \subset \mathcal{A}$, which we will 
define in section \ref{construct2}. 
With $c_{2} > 0$ conveniently small, to be chosen later, let 
\begin{equation*}
N = c^2_{2} \rho^{-\frac{1}{2} ( d + d' - 1 ) }. 
\end{equation*}
Let $\{ \underline{b}^1, \underline{b}^2, \cdots ,\underline{b}^N \}$ be a subset of $\mathcal{A}^2$, and 
write $\underline{b}^i = a^i_1 a^i_2 \ (i = 1, 2, \cdots, N)$ for some $a^i_1$, $a^i_2 \in \mathcal{A}$. 
Then, we say that the words $\underline{b}^1, \underline{b}^2, \cdots ,\underline{b}^N$ are \emph{independent} 
if 
\begin{itemize}
\item[(i)] the words $a_1^1, a_1^2, \cdots, a_1^N$ are mutually distinct;
\item[(ii)] $a^i_1 \in \mathcal{A}_1$, $a^i_2 \in \mathcal{A}_2$ for all $i = 1, 2, \cdots, N$.
\end{itemize}

We then define $\mathcal{L}^0(\rho)$ to be the set of points $t \in \mathbb{R}$ such that $|t| < 1 + s_0$ and the following holds: 
there exist pairs 
$( \underline{b}^1, \underline{b}'^1 ), \cdots, ( \underline{b}^N, \underline{b}'^N )$ in 
$\mathcal{A}^2 \times \mathcal{A}'^2$ such that the words $\underline{b}^1, \cdots, \underline{b}^N$ are independent and if we set 
\begin{equation*}
T_{\underline{b}^i} T'_{\underline{b}'^{i}} (t) = t_i, 
\end{equation*}
then we have $|t_i| \leq 1 + s_0$ for all $1 \leq i \leq N$.

\begin{rem}\label{box_in_box}
Geometrical interpretation of the above definition is the following: 
consider $K \times K'$ in $\mathbb{R}^2$. Let $t \in \mathcal{L}^0(\rho)$, and let 
$( \underline{b}^1, \underline{b}'^1 ), \cdots, ( \underline{b}^N, \underline{b}'^N )$ 
be the associated pairs of words in $\mathcal{A}^2 \times \mathcal{A}'^2$. 
Write $\underline{b}^i = a^i_1 a^i_2$, $\underline{b}'^{i} = a'^{i}_1 a'^{i}_2 \ (i = 1, 2, \cdots, N)$. 
Then $I(a^{i}_1) \times I'(a'^{i}_1)$ are rectangles of size $\rho^{1/2} \times s_0 \rho^{1/2}$, 
and $I( \underline{b}^i ) \times I'( \underline{b}'^i )$ are rectangles of size $\rho \times s_0 \rho$ 
which are inside $I(a^i_1) \times I'(a'^i_1)$. 
The line $\ell_{t}$ is ``very close" to those $N$ rectangles $I( \underline{b}^i ) \times I'( \underline{b}'^i )$. 
\end{rem}

\begin{rem}\label{indep}
We perturb the Cantor set $K$ by perturbing the maps $f_a \ (a \in \mathcal{A}_1)$. 
By this perturbation, $N$ rectangles $I(a^i_1) \times I'(a'^i_1)$ ``independently" moves horizontaly with order $\rho$ 
(recall that the size of the rectangles $I(a^i_1) \times I'(a'^i_1)$ 
is of order $\rho^{1/2}$). Also, since $a^i_2 \in \mathcal{A}_2$, 
the rectangle $I( \underline{b}^i) \times I'( \underline{b}'^i)$ ``does not move relative to $I(a^i_1) \times I'(a'^i_1)$", or in other words, 
``relative position" between $I(a^i_1) \times I'(a'^i_1)$ and $I( \underline{b}^i) \times I'( \underline{b}'^i)$ does not change. 
Note also that the order of the perturbation and the size of $I( \underline{b}^i) \times I'( \underline{b}'^i)$ 
are both $\rho$. 
\end{rem}

In the next section, we will prove the following estimate:
\begin{prop}\label{nice_estimate}
If $c_{2} > 0$ is sufficiently small, there exists $c_{3} > 0$ such that $| \mathcal{L}^{0}(\rho) | > c_{3}$ for all $\rho > 0$. 
\end{prop}

\subsection{Construction of $\mathcal{A}_1, \mathcal{A}_2$ and proof of Proposition \ref{nice_estimate}}\label{construct2}
For $( a, a' ) \in \mathcal{A} \times \mathcal{A}'$, 
We have  
\begin{equation}\label{measure}
c_{4}^{-1} \rho^{\frac{1}{2}(d + d')} \leq \mu \times \mu' ( K(a) \times K'(a') ) \leq 
c_{4} \rho^{\frac{1}{2} (d + d')},
\end{equation}
where $K(a) = K \cap I(a)$ and $K'(a') = K' \cap I'(a')$. 
Write $J( a, a' ) := 
\pi ( I( a ) \times I' ( a' ) )$.
Then 
\begin{equation*}
| J(a, a') | = (1 + s_0) \rho^{1/2}. 
\end{equation*}
We call $(a, a')$ \emph{good} if there are no more than 
$c_{2}^{-1} \rho^{-\frac{1}{2}( d + d' - 1 )}$ intervals $J( \tilde{a}, \tilde{a}' )$
whose centers are distant from the center of $J(a, a')$ by less than $(1 + s_0) \rho^{ 1/2 }$. 
Call $(a, a')$ \emph{bad} if it is not good. We denote the set of good pairs by $\mathcal{G}$, 
and the set of all rectangles of the form 
$I(a) \times I'(a')$ by $\mathcal{B}$.

\begin{rem}\label{why_N}
By (\ref{measure}), $| \mathcal{B} |$ is of order 
$\rho^{-\frac{1}{2}( d + d' )}$, and the size of these rectangles is of order $\rho^{1/2}$. 
Therefore, for any given $(a, a') \in \mathcal{A} \times \mathcal{A}'$, the number of rectangles which are projected to 
``somewhere nearby" to $J(a, a')$ is expected to be the of order 
\begin{equation*}
\rho^{-\frac{1}{2}(d + d')} / \rho^{1/2} = \rho^{-\frac{1}{2}(d + d' - 1)}. 
\end{equation*}
Therefore, if $c_{2}$ is sufficiently small then it is reasonable to expect that most of the pairs $(a, a')$ are good pairs. 
The next lemma shows that this is indeed true.  
\end{rem}

Recall that $\mu_{K, K'}$ has $L^2$-density, where $\mu_{K, K'}$ is the push-forward of the measure $\mu \times \mu'$ under the map $\pi: (x, x') \mapsto x - x'$. We denote the density by $\chi$. Write $c_5 = \left\| \chi  \right\|^2_{L^2}$.

\begin{lem}\label{bad_good}
The number of bad pairs $(a, a')$ is less than 
\begin{equation*}
6 c_2 c^2_{4}  c_{5} (1 + s_0) \rho^{-\frac{1}{2}( d + d' )}. 
\end{equation*}
In particular, if $c_{2}$ is sufficiently small, we have 
\begin{equation}\label{ufufu}
|\mathcal{G}| \geq \frac{15}{16} | \mathcal{B} |.
\end{equation} 
\end{lem}

\begin{proof}

Let $(a, a')$ be bad. Then, by the definition, we have 
\begin{equation*}
\begin{aligned}
\int_{3J( a, a' )} \chi 
&\geq c^{-1}_{4} \rho^{ \frac{1}{2}(d + d')} \cdot c_{2}^{-1} \rho^{ -\frac{1}{2} (d + d' -1) } \\
&= c^{-1}_{2} c^{-1}_{4} \rho^{1/2} = \frac{1}{3} c_{2}^{-1} c_{4}^{-1} (1 + s_0)^{-1} | 3J( a, a' ) |, 
\end{aligned}
\end{equation*}
where $3J(a, a')$ is the interval of the same center as $J(a, a')$ and length $3| J(a, a') |$. 
By the Cauchy-Schwarz inequality, 
\begin{equation*}
\begin{aligned}
\frac{1}{3} c_{2}^{-1} c_{4}^{-1} (1 + s_0)^{-1} | 3 J( a, a' ) | \int_{3 J( a, a' )} \chi  &\leq 
\left( \int_{3 J( a, a' )} \chi  \right)^{2} \\
&\leq | 3 J ( a, a' ) | \int_{3 J( a, a' )} \chi^{2}, 
\end{aligned}
\end{equation*}
and thus 
\begin{equation*}
\int_{3 J( a, a' )} \chi^{2} \geq 
\frac{1}{3} c_{2}^{-1} c_{4}^{-1} (1 + s_0)^{-1}  \int_{3 J( a, a' )} \chi. 
\end{equation*}
Let $J^{*}$ be the union over all bad pairs $( a, a' )$ of the intervals $3J( a, a' )$. One can 
extract a subfamily of intervals whose union is $J^{*}$ and does not cover any point more than twice. 
Then we obtain 
\begin{equation*}
\int_{J^*} \chi^{2}  \geq 
\frac{1}{6} c_{2}^{-1} c_{4}^{-1} (1 + s_0)^{-1}  \int_{J^{*}} \chi.
\end{equation*}
Therefore,   
\begin{equation*}
 \int_{J^{*}} \chi \leq 6 c_2 c_{4} c_{5} (1 + s_0). 
\end{equation*}
As $J^{*}$ contains $J( a, a' )$ for all bad $( a, a' )$, together with (\ref{measure}) 
the estimate of the lemma follows. 
\end{proof}
 

\begin{rem}
The fact that $\mu_{K, K'}$ has $L^{2}$-norm is used in the above lemma. 
If $\mu_{K, K'}$ does not have $L^{2}$-norm, 
then it may happen that the majority of rectangles $I(a) \times I'(a')$ are projected to 
``somewhere close together", resulting that there are very few good rectangles. 
\end{rem}

\begin{rem}\label{long}
 Let $J$ be the union of intervals of the form $J(a, a')$, where $(a, a')$ is a good pair. Then, by Lemma \ref{bad_good}, we have 
\begin{equation}\label{long_time}
\begin{aligned}
|J| &\geq | \mathcal{G} | / ( c_{2}^{-1} \rho^{-\frac{1}{2}(d + d' - 1)} )  \cdot (1 + s_0) \rho^{1/2} \\
&\geq \frac{15}{16} c_4^{-1} \rho^{-\frac{1}{2}( d + d' )} \cdot c_2 \rho^{\frac{1}{2}(d + d' - 1)} \cdot (1 + s_0) \rho^{1/2} 
= \frac{15}{16} c_2 c_4^{-1} (1 + s_0). 
\end{aligned}
\end{equation} 
The heuristic reason of this is that the majority of pairs $(a, a')$ are good pairs, and the set of intervals 
\begin{equation*}
\left\{ J(a, a') : (a, a') \text{ is a good pair} \right\}
\end{equation*}
 do not ``cluster together too much''. 
\end{rem}

By the estimate (\ref{ufufu}), 
we can choose two subsets $\mathcal{A}_1$, $\mathcal{A}_2 \subset \mathcal{A}$ in such a way that 
\begin{itemize}
\item[(i)] for any $a \in \mathcal{A}_1$, $I(a)$ is neither the left nor the right endmost interval;
\item[(ii)] $\displaystyle{ \left| \mathcal{G}^{(\ell)} \right|  \geq \frac{1}{3} |\mathcal{G}| }$, where 
$\displaystyle{ \mathcal{G}^{(\ell)} = \left\{ (a, a') \in \mathcal{G} : a \in \mathcal{A}_{\ell}  \right\} }$ \ $(\ell = 1, 2)$. 
\end{itemize}
For $(a_1, a'_1) \in \mathcal{G}^{(1)}$, we define 
\begin{equation*}
J_{0}( a_1, a'_1 ) = \bigcup_{ (a_2, a'_2) \in \mathcal{G}^{(2)} }  \pi  \left( I( a_1 a_2 ) \times I'( a'_1 a'_2 ) \right). 
\end{equation*}
Then, repeating the argument in Remark \ref{long}, we obtain 
\begin{equation*}
| J_0( a_1, a'_1 ) | \geq c_2 c_{6}  \rho^{ 1/2 }.
\end{equation*}

Let $\varphi$ (resp. $\varphi_0$) be the sum, over $( a, a' ) \in \mathcal{G}^{(1)}$, of the characteristic functions of 
$J( a, a' )$ (resp. $J_0( a, a' )$). 
Note that $\mathrm{supp} \, \varphi, \mathrm{supp} \, \varphi_0 \subset [-(1 + s_0), 1 + s_0]$.  
Since $|\mathcal{G}^{(1)}| \geq c_{7} \rho^{ -\frac{1}{2} (d + d')}$, we have 
\begin{equation*}
\begin{aligned}
\int \varphi_{0} &\geq c_2 c_{6} \rho^{ 1/2 } \cdot c_{7} \rho^{ -\frac{1}{2} (d + d')} \\
&= c_2 c_{8} \rho^{ -\frac{1}{2} (d + d' - 1)}.
\end{aligned}
\end{equation*}
On the other hand, one has
\begin{equation*}
\varphi_0 \leq \varphi \leq c_{2}^{-1} \rho^{ -\frac{1}{2} (d + d' - 1)}. 
\end{equation*}
Let $E = \{  \varphi_0 \geq  c^2_{2} \rho^{ -\frac{1}{2} (d + d' - 1)}  \}$.
Then we have 
\begin{equation*}
\begin{aligned}
|E| \cdot c_{2}^{-1} \rho^{ -\frac{1}{2} (d + d' - 1)} + 
2 (1 + s_0) \cdot  c^2_{2} \rho^{ -\frac{1}{2} (d + d' - 1)} 
&\geq \int_{E} \varphi_0 + \int_{[-(1 + s_0), 1 + s_0] \setminus E} \varphi_0  \\
&= \int \varphi_0 \\
&\geq c_2 c_{8} \rho^{ -\frac{1}{2} (d + d' - 1)}.
\end{aligned}
\end{equation*}
Take $c_2$ small enough so that 
\begin{equation*}
2(1 + s_0) \cdot c^2_2 < \frac{1}{2} c_2 c_8
\end{equation*}
holds. 
Then we obtain 
\begin{equation*}
|E| \geq \frac{1}{2} c_2 c_{8} \cdot c_{2} =: c_{3}. 
\end{equation*}
Since $E \subset \mathcal{L}^0(\rho)$, we have proved that 
\begin{equation*}
| \mathcal{L}^0(\rho)| \geq c_{3}. 
\end{equation*}

\section{Proof of the key Proposition}\label{key_prop}

\subsection{Proof of Proposition \ref{key_lem}}
In this section, we prove Proposition \ref{key_lem}. 
Fix $t \in \mathcal{L}^1(\rho)$. Let $\tilde{t} \in \mathcal{L}^0(\rho)$ be such that $| t - \tilde{t} | < \rho$. 
There exist pairs 
$( \underline{b}^1, \underline{b}'^1 ), \cdots, ( \underline{b}^N, \underline{b}'^N )$ in 
$\mathcal{A}^2 \times \mathcal{A}'^2$ such that the words $\underline{b}^1, \cdots, \underline{b}^N$ are independent, and if we set 
\begin{equation*}
T_{\underline{b}^i} T'_{\underline{b}'^{i}} ( \tilde{t} ) = \tilde{t}_i, 
\end{equation*}
then we have $| \tilde{t}_i| \leq 1 + s_0$ for all $1 \leq i \leq N$. 
Write $\underline{b}^i = a^i_1 a^i_2 \ (i = 1, 2, \cdots, N)$. Let us denote 
\begin{equation*}
\mathcal{A}_3 = \left\{ a^1_1, a^2_1, \cdots, a^N_1 \right\}.
\end{equation*}
Recall that $\mathcal{A}_3 \subset \mathcal{A}_1$. We write 
\begin{equation*}
\begin{aligned}
\Omega &= [-1, 1]^{ \mathcal{A}_3 } \times [-1, 1]^{ \mathcal{A}_1 \setminus \mathcal{A}_3 },  \\
\underline{\omega} &= ( \underline{\omega}', \underline{\omega}'' ), \text{ and \ } 
 \underline{\omega}' = \left( \omega_1, \omega_2, \cdots, \omega_N \right). 
\end{aligned}
\end{equation*}
Notice that, by the construction of $\mathcal{A}_1, \mathcal{A}_2$, as a function of $\underline{\omega}$,  
 $T^{\underline{\omega}}_{ \underline{b}^i } T'_{ \underline{b}'^{i} } (t)$ depends only on 
$\omega_i$, and 
not on $\omega_j \ (j \neq i)$ or $\underline{\omega}''$. We denote 
\begin{equation*}
t_i( \omega_i ) = T^{\underline{\omega}}_{ \underline{b}^i } T'_{ \underline{b}'^{i} } (t).
\end{equation*} 
By Fubini's Theorem, 
Proposition \ref{key_lem} follows from the following claim:
\begin{claim}\label{claim1}
There exists $c''_1 > 0$ such that 
\begin{equation}\label{estimate}
\left|  \left\{ \omega_i : t_i( \omega_i ) \in \mathcal{L}^0(\rho) \right\} \right| \geq c''_1. 
\end{equation}
\end{claim}

\begin{proof}[proof of the claim]
Write $c_9 =  T_{ \underline{b}^i } T'_{ \underline{b}'^{i} } (t) - T_{ \underline{b}^i } T'_{ \underline{b}'^{i} } (\tilde{t})$. 
Note that $|c_9| < | t - \tilde{t} | \cdot \rho^{-1} < \rho \cdot \rho^{-1} = 1$. Since
\begin{equation*}
\begin{aligned}
t_i( \omega_i ) - \tilde{t}_i &= 
\left( T^{\underline{\omega}}_{ \underline{b}^i } T'_{ \underline{b}'^{i} } (t) -  T_{ \underline{b}^i } T'_{ \underline{b}'^{i} } (t) \right) 
+ \left( T_{ \underline{b}^i } T'_{ \underline{b}'^{i} } (t) - T_{ \underline{b}^i } T'_{ \underline{b}'^{i} } (\tilde{t})  \right) \\
&= c_0 \omega_i \rho \cdot \rho^{-1} + c_9 = c_0 \omega_i + c_9, 
\end{aligned}
\end{equation*}
we obtain $t_i( \omega_i ) = c_0 \omega_i + c_{10}$, where $c_{10} = c_9 + \tilde{t}_i$. Since $|c_{10}| < 2 + s_0$, if $c_0$ is sufficiently large 
we have
\begin{equation*}
\begin{aligned}
\left|  \left\{ \omega_i : t_i( \omega_i ) \in \mathcal{L}^0(\rho) \right\} \right| &= | \mathcal{L}^{0}(\rho) | / c_0 \\
&\geq c_3/ c_0. 
\end{aligned}
\end{equation*}
\end{proof}

\section{Proof of Theorem \ref{thm2} and Theorem \ref{thm3}}\label{proof2}
\subsection{Proof of Theorem \ref{thm2}}
The proof is almost the same as that of Theorem \ref{thm1}, but slight modification is necessary. 
This is because (using the notations in section \ref{construct}) we need to avoid that the rectangle  
$I(\underline{b}^i) \times I(\underline{b}'^i)$ moves vertically relative to the rectangle $I(a^i) \times I(a'^i)$. 
See also Remark \ref{indep}.

We modify the definition of independence in the following way (we define $\tilde{\mathcal{A}}_1$ and $\tilde{\mathcal{A}}_2$ later): 
assume that we are given $( \underline{b}^1, \underline{b}'^1 ), \cdots, ( \underline{b}^N, \underline{b}'^N ) \in 
\mathcal{A}^2 \times \mathcal{A}^2$. 
Write $\underline{b}^i = a^i_1 a^i_2, \ \underline{b}'^i = a'^i_1 a'^i_2$ for some 
$a^i_1$, $a^i_2 \in \mathcal{A}$ and $a'^i_1$, $a'^i_2 \in \mathcal{A}$.  
Then, we say that the set of words 
$( \underline{b}^1, \underline{b}'^1 ), \cdots, ( \underline{b}^N, \underline{b}'^N )$ are \emph{independent} 
if
\begin{itemize}
\item[(i)] the words $a^1_1, \cdots, a^N_1$ are mutually distinct;
\item[(ii)] $a^i_1 \in \tilde{\mathcal{A}}_1$ and $a'^i_1, a^i_2, a'^i_2 \in \tilde{\mathcal{A}}_2$ for all $i = 1, 2, \cdots, N$.
\end{itemize}


The definition of $\tilde{ \mathcal{L} }^0(\rho)$ is exactly the same as that of $\mathcal{L}^0(\rho)$: 
namely, we define $\tilde{ \mathcal{L} }^0(\rho)$ to be the set of points $t \in \mathbb{R}$ such that $|t| < 2$ and there exist 
independent pairs 
$( \underline{b}^1, \underline{b}'^1 ), \cdots, ( \underline{b}^N, \underline{b}'^N )$ in 
$\mathcal{A}^2 \times \mathcal{A}^2$ such that, if we set 
\begin{equation*}
T_{\underline{b}^i} T'_{\underline{b}'^{i}} (t) = t_i, 
\end{equation*}
then we have $|t_i| \leq 2$ for all $1 \leq i \leq N$. 

We next define $\tilde{\mathcal{A}}_1, \tilde{\mathcal{A}}_2 \subset \mathcal{A}$.  
Recall that $\mathcal{B}$ is the set of all rectangles of the form $I(a) \times I(a')$, and $\mathcal{G}$ is the set of good pairs. 
Recall also that we have $| \mathcal{G} | \geq \frac{15}{16} | \mathcal{B} |$. We need the following simple lemma: 

\begin{lem}
There exist disjoint subsets $\tilde{ \mathcal{A} }_1, \tilde{ \mathcal{A} }_2 \subset \mathcal{A}$ such that  
\begin{equation*}
| \mathcal{G}^{(\ell m)} | \geq \frac{3}{64} | \mathcal{G} | \ \ ( 1 \leq \ell, m \leq 2), 
\end{equation*}
where 
\begin{equation*}
\mathcal{G}^{(\ell m)} = \left\{ (a, a') : (a, a') \in \mathcal{G}, \, a \in \tilde{ \mathcal{A} }_\ell, \, a' \in \tilde{ \mathcal{A} }_m  \right\}. 
\end{equation*}
\end{lem}

\begin{proof}
For any $a \in \mathcal{A}$, we define a subset $\mathcal{G}_a \subset \mathcal{G}$ to be the union of all good pairs $(a, a')$. 
Write 
\begin{equation*}
\bar{A} = \left\{ a \in \mathcal{A} : | \mathcal{G}_a | \geq \frac{3}{4} |\mathcal{A}| \right\}. 
\end{equation*}
\begin{claim}
We have $| \bar{A} | \geq \frac{3}{4} | \mathcal{A} |$. 
\end{claim}
\begin{proof}[proof of the claim]
Assume that $| \bar{A} | < \frac{3}{4} | \mathcal{A} |$. Then, 
\begin{equation*}
\begin{aligned}
| \mathcal{G} | \leq& | \mathcal{A} | \cdot | \bar{ \mathcal{A} } | + \frac{3}{4} | \mathcal{A} | \cdot  | \mathcal{A} \setminus \bar{ \mathcal{A} } | \\
=& \frac{3}{4} | \mathcal{A} |^2 + \frac{1}{4} | \mathcal{A} | \cdot | \bar{ \mathcal{A} } |
< \frac{15}{16} | \mathcal{A} |^2 = \frac{15}{16} |\mathcal{B}|, 
\end{aligned}
\end{equation*}
which is a contradiction. 
\end{proof}

Let us choose $\tilde{ \mathcal{A} }_1$, $\tilde{ \mathcal{A} }_2$ in such a way that 
$\tilde{ \mathcal{A} }_1 \sqcup \tilde{ \mathcal{A} }_2 = \bar{\mathcal{A}}$ and $| \tilde{ \mathcal{A} }_1 | = | \tilde{ \mathcal{A} }_2 |$. 
Then, since $| \tilde{ \mathcal{A} }_i | \geq \frac{3}{8} |\mathcal{A}|$, 
we obtain 
\begin{equation*}
| \mathcal{G}^{(\ell m)} | \geq \frac{3}{8} | \mathcal{A} | \cdot \left( \frac{3}{8} | \mathcal{A} | - \frac{1}{4} | \mathcal{A} | \right) = \frac{3}{64} | \mathcal{B} |. 
\end{equation*}
The result follows from this.  
\end{proof}

The rest is completely analogous to the proof of Theorem \ref{thm1}. 

\subsection{Proof of Theorem \ref{thm3}}
Note that, for $(K, K') \in \mathcal{Y}$, $\pi(K \times K')$ is a homogeneous self-similar set and $\pi( \mu \times \mu' )$ is its 
uniform self-similar measure. 
Therefore, in the case of homogeneous self-similar sets  
Theorem \ref{thm3} can be proven by repeating the proof of Theorem \ref{thm1}. 

Let us consider the non-homogeneous case. 
Notice that, for any non-homogeneous self-similar set $K$, if we take $c_{11} > 0$ large enough 
then for any $0 < \rho < 1$ there exists a set of contractions $\mathcal{F} = \{ f_a \}_{a \in \mathcal{A}}$ such that  
\begin{itemize}
\item[(i)] $K = \displaystyle{ \bigcup_{a \in \mathcal{A}} f_a(K) }$; 
\item[(ii)] the contracting ratios of $f_a \ (a \in \mathcal{A})$ are all bounded between $c^{-1}_{11}$ and $c_{11}$. 
\end{itemize}
Instead of $\rho$-contractions in the case of homogeneous self-similar sets, we perturb the above set of contractions. The rest of the proof is analogous.

\section*{Acknowledgements}
The author would like to acknowledge the invaluable contributions of Anton Gorodetski. 
The author would also like to thank Boris Solomyak for many helpful comments and discussions, 
and the anonymous referee for the careful reading and all the helpful suggestions and remarks. 
Part of the work on this paper was conducted during the 2017 program ``Fractal Geometry and Dynamics" at Institut Mittag-Leffler. The author is grateful to the organizers and staff for their support and hospitality.


\begin{thebibliography}{00}
\bibitem{Anisca} R.\ Anisca, C.\ Chlebovec, On the structure of arithmetic sums of Cantor sets with constant ratios of dissection, 
\textit{Nonlinearity} \textbf{22} (2009), 2127--2140.
\bibitem{Astels} S.\ Astels, Cantor sets and numbers with restricted partial quotients, \textit{Trans.\ Amer.\ Math.\ Soc.} \textbf{352} (2000), 133--170.
\bibitem{DG1} D.\ Damanik, A.\ Gorodetski, Sums of regular Cantor sets of large dimension and the square Fibonacci Hamiltonian, 
to appear in \textit{J.\ Stat.\ Phys.}
\bibitem{DG} D.\ Damanik, A.\ Gorodetski, B.\ Solomyak, Absolutely continuous convolutions of singular measures and an application to the square 
Fibonacci Hamiltonian, \textit{Duke.\ Math.\ J.} \textbf{164} (2015), 1603--1640.
\bibitem{Dekking} M.\ Dekking, K.\ Simon, B.\ Sz\'ekely, The algebraic difference of two random Cantor sets: the Larsson family, 
\textit{Ann.\ Probab.} \textbf{39} (2011), 549--586. 
\bibitem{Eroglu} K.\ Eroglu, On the arithmetic sums of Cantor sets, \textit{Nonlinearity} \textbf{20} (2007), 1145--1161.
\bibitem{GN} A.\ Gorodetski, S.\ Northrup, On sums of nearly affine Cantor sets, preprint (arXiv: 1510.07008). 
\bibitem{Hall} M.\ Hall, On the sum and product of continued fractions, \textit{Ann.\ of Math.} \textbf{48} (1947), 966--993.
\bibitem{Hochman} M.\ Hochman, On self-similar sets with overlaps and inverse theorems for entropy, \textit{Ann.\ of Math.} \textbf{180} (2014), 773--822.
\bibitem{Honary} B.\ Honary, C.\ Moreira, M.\ Pourbarat, Stable intersections of affine Cantor sets, \textit{Bull.\ Braz.\ Math.\ Soc.} \textbf{36} (2005), 363--378.
\bibitem{Hunt} B.\ Hunt, I.\ Kan, J.\ Yorke, When Cantor sets intersect thickly, \textit{Trans.\ Amer.\ Math.\ Soc.} 
\textbf{339} (1993), 869--888.
\bibitem{KP} R.\ Kenyon, Y.\ Peres, Intersecting random translates of invariant Cantor sets, \textit{Invent.\ Math.} \textbf{104} (1991), 601--629. 
\bibitem{Kraft} R.\ Kraft, Intersections of thick Cantor sets, \textit{Mem.\ Amer.\ Math.\ Soc.} \textbf{97} (1992), 468. 
\bibitem{Kraft3} R.\ Kraft, Random intersections of thick Cantor sets, \textit{Trans.\ Amer.\ Math.\ Soc.} \textbf{352} (2000), 1315--1328.
\bibitem{Mendes} P.\ Mendes, F.\ Oliveira, On the topological structure of the arithmetic sum of two Cantor sets, \textit{Nonlinearity} \textbf{7} (1994), 329--343.
\bibitem{Mora} P.\ Mora, K.\ Simon, B.\ Solomyak, The Lebesgue measure of the algebraic difference of two random Cantor sets, 
\textit{Indag.\ Math.} \textbf{20} (2009), 131--149.
\bibitem{Moreira00} C.\ Moreira, Sums of regular Cantor sets, dynamics and applications to number theory, \textit{Period.\ Math. Hungar.} \textbf{37} (1998),  55--63.
\bibitem{Moreira0} C.\ Moreira, There are no $C^1$-stable intersections of regular Cantor sets, \textit{Acta Math.} \textbf{206} 
(2011), 311--323.
\bibitem{Moreira} C.\ Moreira, E.\ Morales, Sums of Cantor sets whose 
sum of dimensions is close to $1$, \textit{Nonlinearity} \textbf{16} (2003), 1641--1647.
\bibitem{Moreira2} C.\ Moreira, M.\ Mu\~noz, J.\ Letelier, On the topology of arithmetic sums of regular Cantor sets, 
\textit{Nonlinearity} \textbf{13} (2000), 2077--2087.
\bibitem{Moreira3} C.\ Moreira, J.-C.\ Yoccoz, Stable intersections of regular Cantor sets with large Hausdorff dimensions, 
\textit{Ann.\ of Math.} \textbf{154} (2001), 45--96. 
\bibitem{Nazarov} F.\ Nazarov, Y.\ Peres, P.\ Shmerkin, convolutions of Cantor measures without resonance, 
\textit{Israel J.\ Math.} \textbf{187} (2012), 93--116. 
\bibitem{Newhouse} S.\ Newhouse, The abundance of wild hyperbolic sets and nonsmooth stable sets for diffeomorphisms, 
\textit{Inst. Hautes \'Etudes Sci. Publ. Math.} \textbf{50} (1979), 101--151. 
\bibitem{PalisTakens}  J.\ Palis, F.\ Takens, \textit{Hyperbolicity and sensitive chaotic dynamics at homoclinic bifurcations}, 
Cambridge University Press, Cambridge, 1993.
\bibitem{PS} Y.\ Peres, W.\ Schlag, Smoothness of projections, Bernoulli convolutions, and the dimension of exceptions, 
\textit{Duke Math. J.} \textbf{102} (2000), 193--251. 
\bibitem{Peres} Y.\ Peres, P.\ Shmerkin, Resonace between Cantor sets, \textit{Ergodic Theory Dynam.\ Systems} \textbf{29} (2009), 201--221.
\bibitem{Solomyak} Y.\ Peres, B.\ Solomyak, Self-similar measures and intersections of Cantor sets, \textit{Trans.\ Amer.\ Math.\ Soc.} \textbf{350} (1998), 4065--4087.
\bibitem{Pourbarat} M.\ Pourbarat, On the arithmetic difference of middle Cantor sets, preprint (arXiv: 1306.5880).
\bibitem{Pourbarat2} M.\ Pourbarat, Stable intersection of middle-$\alpha$ Cantor sets, \textit{Commun.\ Contemp.\ Math.} 
\textbf{17} (2015). 
\bibitem{Shmerkin0} P.\ Shmerikin, On the exceptional set for absolute continuity of Bernoulli convolutions, 
\textit{Geom.\ Funct.\ Anal.} \textbf{24} (2014), 946--958. 
\bibitem{Shmerkin} P.\ Shmerkin, Projections of self-similar and related fractals, a survey of recent developments, preprint (arXiv: 1501.00875). 
\bibitem{SS} P.\ Shmerkin, B.\ Solomyak, Absolute continuity of self-similar measures, their projections and convolutions, 
\textit{Trans.\ Amer.\ Math.\ Soc.} \textbf{368} (2016), 5125--5151. 
\bibitem{Solomyak1997} B.\ Solomyak, On the measure of arithmetic sums of Cantor sets, \textit{Indag.\ Math. (N.S.)} \textbf{8} (1997), 133--141. 
\bibitem{Varju} P.\ Varj\'u, Absolute continuity of Bernoulli convolutions for algebraic parameters, preprint (arXiv: 1602.00261). 
\end{thebibliography}
\end{document}